\documentclass[11pt]{amsart}

\usepackage{geometry}
\usepackage{amssymb}
\usepackage{amsmath}
\usepackage{mathtools}
\usepackage{xcolor}
\usepackage[colorlinks=true,citecolor=red,linkcolor=blue,urlcolor=red]{hyperref}
\usepackage[all]{xy}
\usepackage{tikz}
\usetikzlibrary{cd}

\newcommand{\CC}{\mathbb{C}}

\newcommand{\cX}{\mathcal{X}}
\newcommand{\cE}{\mathcal{E}}

\newcommand{\cQ}{\mathcal{Q}}

\newcommand{\cT}{\mathcal{T}}
\newcommand{\cA}{\mathcal{A}}
\newcommand{\cB}{\mathcal{B}}
\newcommand{\cC}{\mathcal{C}}
\newcommand{\cG}{\mathcal{G}}

\newcommand{\cO}{\mathcal{O}}
\newcommand{\Om}{\Omega^1_{\cX}(\log\Delta)}
\newcommand{\Qbar}{\overline{\cQ}}
\newcommand{\Qtor}{\cQ_{\mathrm{tor}}}

\DeclareMathOperator{\length}{length}
\DeclareMathOperator{\CH}{CH}

\newtheorem{thm}{Theorem}[section]
\newtheorem{lem}[thm]{Lemma}

\theoremstyle{definition}
\newtheorem{rem}[thm]{Remark}

\begin{document}
	
	\title[Removing a torsion-free hypothesis on Deligne--Mumford stacks]{Removing the torsion-free hypothesis in a positivity theorem on Deligne--Mumford stacks}
	
	\author{Shengyu Hou, Jihao Liu}
	\address{Department of Mathematics, Peking University, No. 5 Yiheyuan Road, Haidian District, Beijing 100871, China}
	\address{Beijing International Center for Mathematical Research, Peking University, No. 5 Yiheyuan Road, Haidian District, Beijing 100871, China}
	\email{liujihao@math.pku.edu.cn}
	
	\date{\today}
	
	\begin{abstract}
		We remove the torsion-free hypothesis from a positivity theorem of Casalaina-Martin and Zhjeqi, answering a question of them. The main result of this paper was obtained using generative AI, particularly ChatGPT 5.5 Pro and the Danus system.
	\end{abstract}
	
	\maketitle
	
	\section{Introduction}\label{sec:introduction}
	
	We work over the field of complex numbers $\CC$. Positivity statements for logarithmic cotangent sheaves turn geometric input into numerical constraints on their quotients. In the stack setting of Casalaina-Martin and Zhjeqi \cite{CMZ26a}, the relevant objects are coherent quotients of tensor powers of the logarithmic cotangent sheaf on a smooth proper integral Deligne--Mumford stack, and the numerical conclusion is the pseudo-effectivity of the first Chern class.
	
	Casalaina-Martin and Zhjeqi \cite{CMZ26a} proved the following positivity theorem, which is \cite[Theorem~4.3]{CMZ26a} and is also referred to there as Theorem~A.
	
	\begin{thm}[{\cite[Theorem~4.3]{CMZ26a}}]\label{thm:CMZ}
		Let $\cX$ be a smooth proper integral Deligne--Mumford stack over $\CC$ with projective coarse moduli space, and
		let $\Delta\subseteq\cX$ be a reduced divisor with at worst normal crossing singularities. Suppose that some positive
		tensor power of $\Omega^1_\cX(\log \Delta )$ contains a subsheaf with big determinant. Then the first Chern class of every
		torsion-free coherent quotient sheaf of every positive tensor power of $\Omega^1_\cX(\log \Delta )$ is pseudo-effective.
	\end{thm}
	
	In \cite[Remark~4.5]{CMZ26a}, the authors explain that the torsion-free hypothesis in Theorem~\ref{thm:CMZ} is imposed only for a technical reason. Let $\cE_m:=(\Omega ^1_\cX(\log \Delta ))^{\otimes m}$. Given an arbitrary coherent quotient $\pi\colon\cE_m\twoheadrightarrow\cQ$, write $\Qtor\subseteq\cQ$ for the torsion subsheaf and $\Qbar:=\cQ/\Qtor$ for the torsion-free quotient. Theorem~\ref{thm:CMZ} applies to $\Qbar$, and the first Chern class of the torsion sheaf $\Qtor$ is an effective codimension-one cycle. Additivity of the first Chern class shows
	$$
	c_1(\cQ)=c_1(\Qtor)+c_1(\Qbar).
	$$
	Thus the desired pseudo-effectivity of $c_1(\cQ)$ follows immediately.
	
	However, in \cite{CMZ26a}, the first Chern class $c_1(\Qtor)$ is well-defined only when $\Qtor$ is a quotient of a torsion-free coherent sheaf (see Section~\ref{subsec:first-Chern-classes}). In \cite[Remark~4.5]{CMZ26a}, Casalaina-Martin and Zhjeqi state that they do not know whether this holds in general. Our first goal is to prove that it holds in the present setting.

	\begin{thm}[=Theorem A]\label{thm:A}
		Let $\cX$ be an integral Deligne--Mumford stack, let $\cE$ be a locally free coherent $\cO_\cX$-module, and let $\pi\colon\cE\twoheadrightarrow\cQ$ be a coherent quotient. Let $\Qtor\subseteq\cQ$ be the torsion subsheaf. Then $\Qtor$ is a quotient of a torsion-free coherent sheaf. In fact, set $\cE':=\pi^{-1}(\Qtor)=\ker\bigl(\cE\xrightarrow{\ \pi\ }\cQ\to\cQ/\Qtor\bigr)$. Then $\cE'$ is torsion-free and the restriction of $\pi$ induces a surjection $\cE'\twoheadrightarrow\Qtor$. 
	\end{thm}
	
	Thus $c_1(\Qtor)$ can be defined. We can therefore remove the torsion-free hypothesis from Theorem~\ref{thm:CMZ}. This is our second main theorem.
	
	\begin{thm}[=Theorem B]\label{thm:B}
		In the setting of Theorem~\ref{thm:CMZ}, the torsion-free hypothesis on the quotient sheaf may be removed. Namely, let $\cX$ be a smooth proper integral Deligne--Mumford stack over $\CC$ with projective coarse moduli space, and
		let $\Delta\subseteq\cX$ be a reduced divisor with at worst normal crossing singularities. Suppose that some positive
		tensor power of $\Omega^1_\cX(\log \Delta )$ contains a subsheaf with big determinant. Then the first Chern class of every coherent quotient sheaf of every positive tensor power of $\Omega^1_\cX(\log \Delta )$ is pseudo-effective.
	\end{thm}
	
	The paper is organized as follows. In Section~2, we recall the elementary facts about coherent sheaves on Deligne--Mumford stacks that will be used below. In Section~3, we prove Theorem~A. In Section~4, we prove Theorem~B.
	
	\begin{rem}
		The main result of this paper was obtained using generative AI, particularly ChatGPT 5.5 Pro and the Danus system, a specialized agent built on the Rethlas system and substantially more capable of conducting fundamental mathematical research. Human verification and polishing were done afterwards. See \cite{Liu+26} and \cite{Ju+26} for detailed introductions to the Danus system and the Rethlas system, respectively. Due to the limitation of generative AI, it is possible that we have missed some related references in the literature, and we welcome any comments from experts.
	\end{rem}
	
	\subsection*{Acknowledgements}
	Jihao Liu was partially supported by the National Key R\&D Program of China \#\allowbreak 2024YFA1014400. He would like to thank the Rethlas team, namely Haocheng Ju, Jiedong Jiang, Shurui Liu, Guoxiong Gao, Yuefeng Wang, Zeming Sun, Bin Wu, Liang Xiao, and Bin Dong, for their contributions to the development of Rethlas and its customized version used for the problem studied in this paper. He would like to thank Ruochuan Liu and Gang Tian for constant support and encouragement.
	
	Shengyu Hou was partially supported by the NSFC Grant (No. 12271384). He would like to thank Yifei Chen for constant support. He would also like to thank the Rethlas team for their contributions to the development of Rethlas and its customized version used for the problem studied in this paper.
	
	\section{Preliminaries}\label{sec:prelim}
	
	We recall some conventions on coherent sheaves and cycle classes on Deligne--Mumford stacks.
	
	\subsection{Torsion subsheaves}
	
	Let $\cX$ be an integral Deligne--Mumford stack, and let $\cG$ be a coherent sheaf on $\cX$. For every irreducible affine étale chart $U\to \cX$, $\cG(U)$ is an $\cO_U(U)$-module. Let $\cG(U)_{\operatorname{tor}}$ be the torsion submodule of $\cG(U)$. These submodules are compatible with restrictions to smaller irreducible affine étale charts, and hence define a coherent subsheaf $\cG_{\operatorname{tor}}\subseteq\cG$. We call $\cG_{\operatorname{tor}}$ the \textbf{torsion subsheaf} of $\cG$. We say that $\cG$ is \textbf{torsion-free} if $\cG_{\operatorname{tor}}=0$. The following lemma follows immediately from the definition.
	\begin{lem}\label{lem:subsheaf-torsion-free}
		Let $\cG'$ be a subsheaf of $\cG$. Then 
		$$\text{$\cG$ is torsion-free}\quad \Rightarrow \quad \text{$\cG'$ is torsion-free}.$$
	\end{lem}
	
	\subsection{First Chern classes}\label{subsec:first-Chern-classes}
	
	Let $\cX$ be a smooth integral Deligne--Mumford stack. We use the notation for first Chern classes adopted in \cite[Section~1.5]{CMZ26b}. Namely, if $\cG$ is a torsion-free coherent sheaf on $\cX$, we define
	$
	\det \cG:=\left(\bigwedge^{\operatorname{rank}\cG}\cG\right)^{\vee\vee}
	$
	and
	$
	c_1(\cG):=c_1(\det\cG)\in \CH_{\dim \cX-1}(\cX).
	$
	More generally, if $\cG$ is a coherent sheaf on $\cX$ which is a quotient of a torsion-free coherent sheaf, then it admits an exact sequence
	$$
	0\to \cG''\to \cG'\to \cG\to 0
	$$
	such that $\cG'$ is torsion-free. By Lemma~\ref{lem:subsheaf-torsion-free}, the sheaf $\cG''$ is also torsion-free. We define
	$\det \cG:=\det \cG'\otimes (\det\cG'')^{-1}$ and 
	$
	c_1(\cG):=c_1(\det \cG)=c_1(\cG')-c_1(\cG'')
	$.
	This definition is independent of the choice of $\cG'$ and is additive in short exact sequences. More precisely, if
	$$
	0\to \cA\to \cB\to\cC\to 0
	$$
	is a short exact sequence of coherent sheaves which are quotients of torsion-free coherent sheaves, then
	$
	c_1(\cB)=c_1(\cA)+c_1(\cC).
	$
	Thus, whenever we write $c_1(\cG)$ for a coherent sheaf $\cG$, we first need to know that $\cG$ is a quotient of a torsion-free coherent sheaf.
	
	\begin{lem}\label{lem:c1-properties}
		Let $\cX$ be a smooth integral Deligne--Mumford stack, and let $\cT$ be a coherent torsion sheaf which is a quotient of a torsion-free coherent sheaf. Then
		$$c_1(\cT)=\sum_D \operatorname{length}_{\cO_{\cX,\eta_D}}(\cT_{\eta_D})[D],$$
		where $D$ runs over the integral codimension-one closed substacks of $\cX$.
	\end{lem}
	\begin{proof}
		$\cT$ admits an exact sequence $$0\to \cG'\xrightarrow{\varphi } \cG\to \cT\to 0$$ such that $\cG$ and $\cG'$ are torsion-free and $c_1(\cT)=c_1(\det\cG\otimes (\det \cG')^{-1})$. The determinant $\det\varphi$ is a rational section of $\det\cG\otimes (\det \cG')^{-1}$, thus $c_1(\cT)=\sum_D v_D(\det\varphi)[D]$. Let $D\subseteq\cX$ be an integral codimension-one closed substack, and let $\eta_D$ be its generic point. Since $\cX$ is smooth and integral, $\cO_{\cX,\eta_D}$ is a discrete valuation ring and $v_D$ is the valuation on it. Let $\pi$ be the uniformizer of $\cO_{\cX,\eta_D}$. The localizations of $\cG$ and $\cG'$ are free modules of the same rank, thus the localization of the exact sequence becomes $$0\to \cO_{\cX,\eta_D}^r\xrightarrow{\varphi_{\eta_D} } \cO_{\cX,\eta_D}^r\to \cT_{\eta_D}\to 0,$$ where $\varphi_{\eta_D}$ is an $r\times r$ matrix. Using the Smith normal form over DVRs, after changing bases we may assume $\varphi_{\eta_D} $ is a diagonal matrix $\operatorname{diag}(\pi^{a_1}, \cdots , \pi ^{a_r})$. 
		In this case, $\cT_{\eta_D}\cong \bigoplus_{i=1}^r \cO_{\cX,\eta_D}/(\pi^{a_i})$ and $\length_{\cO_{\cX,\eta_D}}(\cT_{\eta_D})=a_1+\cdots+a_r$. The valuation $v_D(\det\varphi )=v_\pi(\det\varphi _{\eta_D})=v_\pi(\pi^{a_1+\cdots +a_r})=a_1+\cdots+a_r$, and thus it is equal to $\length_{\cO_{\cX,\eta_D}}(\cT_{\eta_D})$. 
	\end{proof}
	
	\subsection{Effective and pseudo-effective cycles}
	
	Let $\cX$ be a smooth integral Deligne--Mumford stack, and set $d:=\dim\cX$. We say that a cycle class $\alpha\in \CH_{d-1}(\cX)$ is \textbf{effective} if
	$
	\alpha=\sum_{j=1}^r a_j[D_j]
	$
	for some integral codimension-one closed substacks $D_j\subseteq\cX$ and some $a_j\in\mathbb{N}$.
	
	Moreover, assume that $\cX$ is proper and has a projective coarse moduli space $\pi\colon \cX\to X$. We define pseudo-effectivity as in \cite[Section~2]{CMZ26b}. The push-forward $\pi_*$ induces an isomorphism
	$$
	\pi_*\colon \CH_{d-1}(\cX)_{\mathbb{R}}\xrightarrow{\sim}\CH_{d-1}(X)_{\mathbb{R}}.
	$$
	We say that $\alpha\in \CH_{d-1}(\cX)$ is \textbf{pseudo-effective} if
	$
	\pi_*\alpha\in \CH_{d-1}(X)_{\mathbb{R}}
	$
	is pseudo-effective on the projective variety $X$. With this convention, every effective class is pseudo-effective.
	
	\section{Proof of Theorem A}\label{sec:torsion-quotient}
	
	We first prove Theorem A (=Theorem~\ref{thm:A}).
	
	\begin{lem}\label{lem:inverse-image}
		Let $\cX$ be a noetherian Deligne--Mumford stack, let $\pi\colon\cE\twoheadrightarrow\cQ$ be a surjection of coherent sheaves, and let $\cQ'\subseteq\cQ$ be a coherent subsheaf. Define
		$$
		\cE':=\pi^{-1}(\cQ')=\ker(\cE\xrightarrow{\pi}\cQ\to\cQ/\cQ').
		$$
		Then $\pi|_{\cE'}:\cE'\to \cQ'$ is also surjective. 
	\end{lem}
	
	\begin{proof}
		The category of coherent sheaves on a noetherian Deligne--Mumford stack is abelian. There is a diagram
		$$\begin{tikzcd}
			0 \arrow[r] & \cE' \arrow[r] \arrow[d, "\pi|_{\cE'}"] & \cE \arrow[r] \arrow[d, "\pi"] & \cQ/\cQ' \arrow[r] \arrow[d, "\operatorname{id}"] & 0 \\
			0 \arrow[r] & \cQ' \arrow[r]                          & \cQ \arrow[r]                  & \cQ/\cQ' \arrow[r]                                 & 0
		\end{tikzcd}$$
		The horizontal lines are exact sequences. By the snake lemma, $\pi|_{\cE'}$ is surjective. 
	\end{proof}
	
	\begin{proof}[Proof of Theorem A (=Theorem~\ref{thm:A})]
		Set $\cE':=\pi^{-1}(\Qtor)=\ker(\cE\xrightarrow{\pi}\cQ\to \cQ/\Qtor)$. Lemma~\ref{lem:inverse-image} gives a surjection $\cE'\twoheadrightarrow\Qtor$. Since $\cE$ is locally free, it is torsion-free. Lemma~\ref{lem:subsheaf-torsion-free} implies that $\cE'$ is also torsion-free. Thus $\Qtor$ is a quotient of the torsion-free coherent sheaf $\cE'$.
	\end{proof}
	
	\section{Proof of Theorem B}\label{sec:remove}
	
	This section proves Theorem B (=Theorem~\ref{thm:B}). 
	
	\begin{proof}[Proof of Theorem B (=Theorem~\ref{thm:B})]
		Set $\cE_m:=(\Om)^{\otimes m}$. Let $\pi\colon\cE_m\twoheadrightarrow\cQ$ be a quotient. Let $\Qtor$ be the torsion subsheaf of $\cQ$ and $\Qbar=\cQ/\Qtor$. $\Qbar$ is torsion-free. The composite $\cE_m\twoheadrightarrow\cQ\twoheadrightarrow\Qbar$ is still a quotient. By Theorem~\ref{thm:CMZ}, the class $c_1(\Qbar)$ is pseudo-effective. By Theorem~\ref{thm:A}, $c_1(\Qtor)$ is well-defined. By Lemma~\ref{lem:c1-properties}, the class $c_1(\Qtor)$ is effective. The short exact sequence $0\to\Qtor\to\cQ\to\Qbar\to0$ shows
		$$
		c_1(\cQ)=c_1(\Qtor)+c_1(\Qbar).
		$$
		Since the pseudo-effective cone is closed under addition, $c_1(\cQ)$ is pseudo-effective. This proves the theorem.
	\end{proof}

\end{document}